\documentclass{amsart}

\usepackage[utf8]{inputenc}
\usepackage[hidelinks]{hyperref}
\usepackage[OT2,T1]{fontenc}
\DeclareSymbolFont{cyrletters}{OT2}{wncyr}{m}{n}
\DeclareMathSymbol{\Sha}{\mathalpha}{cyrletters}{"58}

\usepackage{amsmath}
\usepackage{amssymb}

\usepackage{amsthm}
\theoremstyle{plain}
\newtheorem{theorem}{Theorem}[section]
\newtheorem{lemma}[theorem]{Lemma}
\newtheorem{proposition}[theorem]{Proposition}
\newtheorem{corollary}[theorem]{Corollary}
\theoremstyle{definition}
\newtheorem{definition}{Definition}[section]
\newtheorem{conjecture}{Conjecture}
\newtheorem{example}{Example}[section]
\theoremstyle{remark}
\newtheorem{remark}[theorem]{Remark}

\usepackage{tikz-cd}

\title{Topology of projective Tate-Shafarevich twists}
\author{David Zhiyuan Bai}
\date{\today}
\address{Yale University}
\email{david.bai@yale.edu}

\begin{document}
\begin{abstract}
    A Tate-Shafarevich twist $X^\phi\to B$ of a fibration $X\to B$ modifies it by a $1$-cocycle of flows of vector fields relative to the base, locally in the analytic topology.
    Sacc\`a conjectured that the total spaces of two projective Lagrangian fibrations related by such a twist are deformation-equivalent.

    Assuming that the class of the twist is torsion (which is often equivalent to the twist being realizable in the \'etale topology), we show that there is an isomorphism $H^\ast(X;\mathbb Q)\cong H^\ast(X^\phi;\mathbb Q)$ of graded vector spaces that respects (1) the Hodge structures and (2) the Hodge-Riemann pairing.
    Consequently, the rational Beauville-Bogomolov-Fujiki lattices of these two spaces are Hodge-similar.

    Assuming further that $B$ is smooth, and both the original fibration and its twist admit $C^\infty$-sections, we show Sacc\`a's conjecture using the theory of degenerate twistor deformations.
\end{abstract}
\maketitle
\tableofcontents

\section{Introduction}
We work over $\mathbb C$ and, unless otherwise stated, endow the spaces we consider with the analytic (Euclidean) topology.
\subsection{A conjecture of Sacc\`a}
Given a fibration $\pi:X\to B$ and a $1$-cocycle $\phi$ of relative automorphisms that are locally given by flows of vertical vector fields (\emph{i.e.}~sections of $\pi_\ast T_{X/B}$), one can re-glue the fibration with transition functions modified by $\phi$.
This is called the \emph{Tate-Shafarevich twist} of $X$ by $\phi$.
See \S\ref{ss:rec_TS} for a precise formulation and discussion of this construction.

In general, there is no reason for $X^\phi$ to even be homeomorphic to $X$ (Example \ref{example_projbundle}), even though the fibrations have isomorphic fibers.
Nonetheless, when $X$ is an irreducible hyper-K\"ahler variety (in which case its fibrations are called Lagrangian fibrations), the situation becomes very interesting:
It turns out that any twist of a fibration on $X$ is still holomorphic symplectic, hence hyper-K\"ahler (in fact, irreducible hyper-K\"ahler) as soon as it is K\"ahler.

There is only a small set of known deformation types of irreducible hyper-K\"ahler varieties.
One then wonders if it would ever be possible to reach new deformation types by performing a Tate-Shafarevich twist on a Lagrangian fibration.

\begin{conjecture}\label{saccaconj}
    Total spaces of projective Lagrangian fibrations that are related by a Tate-Shafarevich twist are deformation-equivalent.
\end{conjecture}

Sacc\`a considered this in the following context:
In \cite[Theorem 6.13]{sacca}, she started with a Lagrangian fibration with integral fibers and constructed the relative compactified Albanese of it, which is a Tate-Shafarevich twist of the original fibration that admits a section.
From there, she conjectured:
\begin{conjecture}\label{intro:saccaconjorig}
    The relative compactified Albanese of a Lagrangian fibration with integral fibers stays in the same deformation class.
\end{conjecture}
She provided evidence for this by showing that the singular cohomology groups of these two spaces coincide as $\mathbb Q$-Hodge structures by applying Ng\^o's support theorem \cite{ngo}.

The statement of Conjecture \ref{saccaconj} is a generalization of this that relaxes both the condition on fibers (allowing non-integral fibers) and the topology (analytic instead of \'etale) in question.

\subsection{Cohomology groups}
We generalize the observation of Sacc\`a on singular cohomology groups to arbitrary (projective) Tate-Shafarevich twists.

From now on the base $B$ will always be assumed to be normal and projective.
\begin{theorem}\label{intro:hodgestr}
    Suppose $X\to B$ is a projective fibration and $X^\phi\to B$ is a Tate-Shafarevich twist of it.
    Suppose in addition that $X^\phi$ is also projective.
    Then we have $H^k(X;\mathbb Q)\cong H^k(X^\phi;\mathbb Q)$ as $\mathbb Q$-Hodge structures for each $k$.
\end{theorem}
This is proved in \S\ref{ss:pervsheaves}.
We do not need the fibration to be Lagrangian.

Note that this theorem is false if the projectivity assumption on $X^\phi$ is dropped.
Indeed, in \cite{abashevaii}, it was observed that the second Betti number can be different for non-K\"ahler twists of a Lagrangian fibration.
An explicit example given by the Bogomolov-Guan manifold is described in \cite[Remark 5.3.13]{abashevaii}.

Theorem \ref{intro:hodgestr} holds under the projectivity assumption because it allows us to use the decomposition theorem for perverse sheaves.
We in fact show:
\begin{theorem}\label{intro:mixedhodge}
    Under the assumptions of Theorem \ref{intro:hodgestr}, there exists natural isomorphisms ${}^{\mathfrak p}\mathcal H^k(\mathbf R\pi_\ast\underline{\mathbb Q})\cong {}^{\mathfrak p}\mathcal H^k(\mathbf R\pi_\ast^\phi\underline{\mathbb Q})$ of mixed Hodge modules for each $k$.
\end{theorem}
This yields a non-canonical isomorphism $\mathbf R\pi_\ast\underline{\mathbb Q}\cong\mathbf R\pi^\phi_\ast\underline{\mathbb Q}$, from which the previous theorem follows.

\subsection{Lagrangian fibrations}
Let us now impose two conditions for the rest of our theorems.
Firstly, we would like the original fibration to be projective and Lagrangian, \emph{i.e.}~the total space to be an irreducible hyper-K\"ahler variety.
A consequence of this is that a general fiber of the fibration would be an abelian variety, thus the sheaf $\underline{\operatorname{Aut}}_{X/B}^0$ is actually a sheaf of commutative groups.
It then makes sense to define the Tate-Shafarevich group $\Sha_{X/B}=H^1(B,\underline{\operatorname{Aut}}^0_{X/B})$ which parameterizes Tate-Shafarevich twists.

The second condition we impose is that the twist $\phi\in\Sha_{X/B}$ is torsion.
This is expected to be the case for when the twist is algebraic, \emph{i.e.}~it can be achieved in the \'etale topology.
See \S\ref{ss:torsion} for a discussion of this condition.

As a consequence, results in \cite{abashevaii} would show that the twist $X^\phi$ is also a projective Lagrangian fibration.

In this instance, we obtain the invariance of more structures on the singular cohomology.
\begin{theorem}\label{intro:HRform}
    Suppose $\pi:X\to B$ is a projective Lagrangian fibration and $\phi\in\Sha_{X/B}$ is torsion.
    Then there exists ample classes $\omega,\omega^\phi$ on $X,X^\phi$ respectively such that the isomorphism in Theorem \ref{intro:hodgestr} can be chosen to preserve the Hodge-Riemann bilinear form
    \[\langle \alpha,\beta\rangle_\omega=\int(1+\omega+\omega^2+\cdots)\smile\alpha\smile\beta\]
    associated to $\omega$ and $\omega^\phi$.
\end{theorem}

This is obtained by studying, instead of the cohomology ring, the associated graded ring with respect to the perverse filtration.

Suppose $\pi:X\to B$ is a fibration.
The perverse filtration on the singular cohomology of $X$ is defined as the images of
\[P_kH^l(X;\mathbb Q)=\operatorname{Im}(H^l(B,{}^{\mathfrak p}\tau_{\le k+\dim B}\mathbf R\pi_\ast\underline{\mathbb Q})\to H^l(B,\mathbf R\pi_\ast\underline{\mathbb Q})=H^l(X;\mathbb Q).\]

Perverse filtrations for generically abelian fibrations have been studied intensively in the last decade.
Among many other examples, they are used in the study of the Hitchin fibration in relation to the $P=W$ conjecture \cite{pwconj, pw1, hmms, pw3}, moduli spaces of one-dimensional sheaves on del Pezzo surfaces \cite{weite1, weite2, weite3}, Lagrangian fibrations of compact hyper-K\"ahler manifolds \cite{shenyin,shenyinharderli,shenyinfelisetti}, and universal compactified Jacobians \cite{bmsy25}.

The perverse filtration is called multiplicative if the cup product takes $P_kH^l\times P_{k'}H^{l'}$ to $P_{k+k'}H^{l+l'}$.
If this were the case, then the cup product descends to a product on the bigraded vector space
\[\mathbb H_\pi=\bigoplus_{k,l}\operatorname{Gr}_k^PH^l(X;\mathbb Q)\]
making it a bigraded ring.
We call this the associated graded ring.

In general the associated graded ring is far from being isomorphic to the original cohomology ring.
However, the miracle in the case of a Lagrangian fibration is that not only is the perverse filtration always multiplicative, but there is actually a multiplicative splitting of the perverse filtration, \emph{i.e.}~we have an isomorphism
\[H^\ast(X;\mathbb Q)\cong\mathbb H_\pi\]
of graded rings (only the cohomology grading is considered on the latter).
This is the result of \cite[Appendix A]{shenyin}.

Our strategy is then to show Theorem \ref{intro:HRform} for the analogous construction on $\mathbb H_\pi$.
This is easier because of the following:
The cup product on the level of complexes
\[\mathbf R\pi_\ast\underline{\mathbb Q}\otimes\mathbf R\pi_\ast\underline{\mathbb Q}\to\mathbf R\pi_\ast\underline{\mathbb Q}\]
is impossible to control using only local information (generally, maps between complexes are not sheafy, \emph{i.e.}~determined by values on an open cover).
However, if we fix a splitting
\[\mathbf R\pi_\ast\underline{\mathbb Q}\cong\bigoplus\mathcal H_i[-i]\]
with $\mathcal H_i\in\operatorname{Perv}(B)[-\dim B]$, then the cup product would induce maps $\mathcal H_k\otimes\mathcal H_{k'}\to\mathcal H_{k+k'}$.
We show that after taking global cohomology, this recovers the product
\[\operatorname{Gr}_k^PH^\ast (X;\mathbb Q)\times\operatorname{Gr}_{k'}^PH^\ast (X;\mathbb Q)\to \operatorname{Gr}_{k+k'}^PH^\ast (X;\mathbb Q) \]
on $\mathbb H_\pi$.

Using a construction in \cite{abashevaii}, we show:
\begin{theorem}
    There exists splittings of the respective perverse filtrations such that the maps $\mathcal H_k\otimes\mathcal H_{k'}\to\mathcal H_{k+k'}$ coincide for $X$ and $X^\phi$ (with the identifications given by Theorem \ref{intro:mixedhodge}) provided that one of $\mathcal H_k,\mathcal H_{k'},\mathcal H_{k+k'}$ is a local system.
\end{theorem}
This is shown in \S\ref{ss:Hpi_twist} for a class of fibrations and twists satisfying certain conditions.
We explain in \S\ref{ss:rec_lagfib} why our setup (projective Lagrangian fibrations under torsion twists) is in this class.

Consequently, the intersection pairing as well as the action of $H^\ast(B;\mathbb Q)=H^\ast(\mathcal H_0)$ are preserved under the isomorphism $\mathbb H_\pi\cong\mathbb H_{\pi^\phi}$ of vector spaces.

We show in \S\ref{ss:lefschetz} that certain actions by relatively ample classes are also preserved under this isomorphism.
When combined, they show Theorem \ref{intro:HRform} as explained in \S\ref{ss:specialize_lag}.

As a consequence of Theorem \ref{intro:HRform}, we see that
\begin{corollary}
    There is a Hodge similitude $(H^2(X;\mathbb Q),q_X)\to (H^2(X^\phi;\mathbb Q),q_{X^\phi})$ where $q$ is the Beauville-Bogomolov-Fujiki (BBF) form.
\end{corollary}
A natural question is whether we can compare the integral BBF lattices.
This turns out to be hard:
Indeed, the map in the corollary cannot preserve the integral lattices in general.
Nevertheless, we discuss in \S\ref{ss:bbf} some partial results in this front.
\subsection{Fibrations with \texorpdfstring{$C^\infty$}{Cinfty}-sections}
Let us add further constraints on our fibration.
From now on, we assume in addition that $B$ is smooth.
This is known to be true if $\dim B=2$ (see \cite{ou19,BK18,HX20}), and has been conjectured to be the case in general \cite{huybrechtsmauri}.

We say a fibration $\pi:X\to B$ between smooth varieties admits $C^\infty$-sections if there is a $C^\infty$-map $s:B\to X$ between the underlying real manifolds such that $\pi\circ s=\operatorname{id}_B$.
\begin{theorem}\label{intro:withsection}
    Suppose both $X\to B$ and $X^\phi\to B$ admit $C^\infty$-sections.
    Then Conjecture \ref{saccaconj} is true.
\end{theorem}
\begin{remark}
    In \cite{bkv}, it was asserted that Lagrangian fibrations admit $C^\infty$-sections if it is reduced in codimension $1$, and a general fiber of it has primitive integral homology class.
    It was pointed out to us by Koll\'ar and Sacc\`a that the proof in that manuscript does not go through.
\end{remark}

The main ingredient is a theorem of Bogomolov-D\'eev-Verbitsky \cite{BDV} on degenerate twistor deformations.
Their theorem says that, given any Lagrangian fibration admitting a $C^\infty$-section, there is a member in its degenerate twistor family making this section holomorphic.

Starting from our setup, we perform such a deformation on $X$ and $X^\phi$ respectively.
We show that the resulting spaces are still projective Lagrangian fibrations related by a torsion twist, which in the situation in question implies that the twist can be realized in the \'etale topology.

Since the new spaces now admit sections, we can produce open subschemes in them that are isomorphic by comparing the abelian schemes over a suitable open subscheme of the base.
As birational hyper-K\"ahler varieties are deformation-equivalent, we obtain Theorem \ref{intro:withsection}.
The details can be found in \S\ref{ss:twistsection}.

For now, we are unable to decide whether every Lagrangian fibration with integral fibers admits $C^\infty$-sections.
If this were true, then Conjecture \ref{intro:saccaconjorig} follows from Theorem \ref{intro:withsection}.
\begin{remark}
    Note that there are plenty of examples of Lagrangian fibrations that do not admit any $C^\infty$-section.
    Indeed, if a fibration admits a $C^\infty$-section, then it cannot have any nowhere-reduced fiber.
    Furthermore, the integral homology class of a general fiber of it must be primitive, \emph{i.e.}~indivisible.
    Now suppose $S$ is a K3 surface and $\ell$ is a primitive, basepoint-free, big and nef class on $S$.
    Then for any $r\in\mathbb Z_{\ge 1},\chi\in\mathbb Z$ coprime and a generic polarization, the moduli space $M$ of Gieseker-stable sheaves $\mathcal F$ on $S$ with
    \[c_1(\mathcal F)=r\ell,\chi(\mathcal F)=\chi\]
    is an irreducible hyper-K\"ahler manifold of K3${}^{[n]}$-type.
    There is a Lagrangian fibration $M\to |r\ell|$ taking such a sheaf to its Fitting support.
    And it has been shown in \cite{baifang} that a general fiber of this fibration has primitive integral homology class if and only if $r=1$.
    In fact, in the case $r>1$, it is shown in \emph{loc.~cit.}~that the fibration has a nowhere-reduced fiber.
\end{remark}

\subsection{Related work}
Tate-Shafarevich twists for Lagrangian fibrations can be formulated in various different ways.
We closely follow the constructions of Abasheva and Rogov \cite{AR21} (see also \cite{abashevaii}).
Recent work of Kim \cite{kim_neron} allows one to study Tate-Shafarevich twists that take value in various group schemes.
Koll\'ar \cite{kollar} studied, more generally, fibrations whose general fiber is an abelian variety, and showed the existence of ``untwisting'' under certain conditions.

For Lagrangian fibrations coming from K3 surfaces or cubic fourfolds, Tate-Shafarevich twists and degenerate twistor deformations can be studied concretely, often in relation to Brauer groups.
See the results of Markman \cite{markman2013}, Huybrechts-Mattei \cite{huybrechtsmattei1,huybrechtsmattei2}, and Dutta-Mattei-Shinder \cite{duttamatteishinder}.

\subsection*{Acknowledgements}
The author is grateful to Junliang Shen for fruitful discussions and for explaining the results of \cite{shenyin}, and to Giulia Sacc\`a and J\'anos Koll\'ar for useful comments, particularly regarding \cite{bkv}.
The author would also like to thank David Fang, Weite Pi, and Vladyslav Zveryk for helpful comments.

The author was partially supported by NSF grant DMS-2301474.

\section{Tate-Shafarevich twists and the decomposition theorem}
\subsection{Recollections on Tate-Shafarevich twists}\label{ss:rec_TS}
We are going to be interested in analytic Tate-Shafarevich twists as formulated in \cite{abashevaii}.

For us, a fibration is a proper holomorphic surjection $\pi:X\to B$ with connected fibers, where $X$ is a connected complex manifold, $B$ is a normal quasiprojective variety, and $0<\dim B<\dim X$.
\begin{definition}
    For a fibration $\pi:X\to B$, the sheaf $T_{X/B}$ of vertical vector fields is defined to be the kernel of $T_X\to \pi^\ast T_B/\operatorname{Tors}(\pi^\ast T_B)$, where $\operatorname{Tors}(\pi^\ast T_B)$ is the torsion subsheaf of $\pi^\ast T_B$.

    Denote by $\underline{\operatorname{Aut}}_{X/B}^0$ the image of the exponential map $\exp:\pi_\ast T_{X/B}\to \underline{\operatorname{Aut}}_{X/B}$.
\end{definition}
So $\underline{\operatorname{Aut}}_{X/B}^0$ consists of those automorphisms that are flows of vertical vector fields.
Note also that $\operatorname{Tors}(\pi^\ast T_B)=0$ if $B$ is smooth.

Given a $1$-cocycle $\phi=\{\phi_{ij}\}$ with values in $\underline{\operatorname{Aut}}^0_{X/B}$ over an open cover $\{U_i\}$ (in the analytic topology) of $B$, one can form a new fibration $\pi^\phi:X^\phi\to B$ by re-gluing $X_i=\pi^{-1}U_i$ along the automorphisms $\phi_{ij}$.

Clearly, $\pi^\phi$ is independent of modifications by a coboundary.
\begin{definition}
    We call $\pi^\phi:X^\phi\to B$ the Tate-Shafarevich twist of $\pi$ by $\phi$.
\end{definition}
\begin{example}\label{example:compactified_jac_twist}
    Suppose $c:\mathcal C\to B$ is a family of integral curves, then one can form the compactified Jacobian fibrations $X=\overline{\operatorname{Pic}}^d(\mathcal C/B)\to B$ of different degrees $d$.
    Let us assume that these spaces $\overline{\operatorname{Pic}}^d(\mathcal C/B)$ are smooth.

    All of them are related to each other via a Tate-Shafarevich twist.
    Indeed, the exponential map $\pi_\ast T_{X/B}\to\underline{\operatorname{Aut}}_{X/B}$ in this case factors through
    \[\mathbf R^1c_\ast\mathcal O\to \mathbf R^1c_\ast\mathcal \mathcal O^\times.\]

    On the other hand, these families of compactified Jacobians of different degrees are isomorphic as soon as the curve family admits a section.
    On the other hand, an arbitrary curve family of such kind always admits analytic local sections.
    So these compactified Jacobians are isomorphic when restricted to an analytic open covering $\{U_i\}$ of the base.

    On double intersections $U_i\cap U_j$, the difference between the choices of an analytic section on $U_i$ versus that on $U_j$ is measured precisely by an action of $\operatorname{Pic}_{\mathcal C_{U_i\cap U_j}/U_i\cap U_j}^0$, which lies in $\Gamma(U_i\cap U_j,\underline{\operatorname{Aut}}_{X/B}^0)$.
\end{example}
\begin{example}\label{example_projbundle}
    Consider $X=B\times\mathbb P^{r-1}\to B$.
    Note that in this case $\pi_\ast T_{X/B}\to\underline{\operatorname{Aut}}_{X/B}$ is simply the exponential map
    \[\mathfrak{pgl}_r\to\operatorname{PGL}_r,\]
    and is therefore surjective.
    So any $\mathbb P^{r-1}$-bundle $\mathbb P\to B$ is a Tate-Shafarevich twist of $X\to B$.

    Note that $\mathbb P$ is generally not even homeomorphic to $X$ (by looking at the singular cohomology ring, for example).
    Nevertheless, they do share the same Betti numbers.
    We will further discuss this phenomenon soon.
\end{example}
\begin{definition}
    The Tate-Shafarevich group for the fibration $\pi:X\to B$ is $\Sha_{X/B}=H^1(B,\underline{\operatorname{Aut}}^0_{X/B})$.
\end{definition}
\begin{remark}
    As shown by Example \ref{example_projbundle}, $\underline{\operatorname{Aut}}^0_{X/B}$ can be noncommutative, in which case $\Sha_{X/B}$ is to be understood in the framework of nonabelian cohomology, and is not a group in general.
    That said, $\underline{\operatorname{Aut}}^0_{X/B}$ is commutative whenever a generic fiber of the fibration is a complex torus, which is the case for our main situation of interest (namely, Lagrangian fibrations).
\end{remark}
\begin{example}\label{ex:elliptick3}
    Suppose $S\to B=\mathbb P^1$ is an elliptic K3 surface with integral fibers and let $X=\overline{\operatorname{Pic}}^0(S/B)$.
    Then the map $H^1(B,\pi_\ast T_{X/B})\to H^1(B,\underline{\operatorname{Aut}}_{X/B}^0)=\Sha_{X/B}$ is actually surjective.
    In fact, as computed in \cite[Ch.~11]{k3},
    \[\Sha_{X/B}\cong\operatorname{Br}(X)\cong\mathbb C/\mathbb Z^{22-\rho(X)}\]
    where $\rho(X)$ is the Picard rank of $X$.

    As an aside:
    Take any $p\in B$ and write $B^0=B\setminus\{p\}$ and $X^0=\pi^{-1}B^0$.
    Then $H^1(B^0,\pi_\ast T_{X^0/B^0})=0$, which means that $X^0\cong(\pi^\phi)^{-1}(B^0)$ as complex manifolds for any $\phi\in\Sha_{X/B}$.
    But of course a Tate-Shafarevich twist of $X$ is generally not birational to $X$:
    As exhibited in Example \ref{example:compactified_jac_twist}, one can reach $\overline{\operatorname{Pic}}^d(S/B)$ for any $d$, and they are usually not all birational to each other (unless $S\to B$ has a section).

    This gives a simple example of a non-proper complex manifold that admits multiple non-isomorphic algebraic structures.
\end{example}

\begin{remark}\label{ts_twice}
    Suppose $\underline{\operatorname{Aut}}^0_{X/B}$ is commutative.
    For any $\phi\in\Sha_{X/B}$, we can naturally identify $\underline{\operatorname{Aut}}_{X/B}^0$ with $\underline{\operatorname{Aut}}_{X^\phi/B}^0$, which gives an identification of $\Sha_{X/B}$ with $\Sha_{X^\phi/B}$.
    Under these isomorphisms, for any $\psi\in\Sha_{X/B}$, we have $(X^\phi)^\psi=X^{\phi+\psi}$.
\end{remark}

\subsection{Perverse sheaves}\label{ss:pervsheaves}
In \cite{abashevaii}, in the occasion of a Lagrangian fibration, the second Betti number of $X$ and $X^\phi$ are compared.
An observation made there is that even when $X$ is a Lagrangian fibration, $X^\phi$ can have a second Betti number different from that of $X$, as a result of a mismatch between the differentials in the Leray spectral sequence.
We refer the reader to \cite[Remark 5.3.13]{abashevaii} for an explicit example.

Our purpose in this section is to point out that the situation improves drastically when a projectivity assumption is placed on both the fibration and its twist.
The difference it makes is the decomposition theorem, which shows that the \emph{perverse} Leray spectral sequence takes a very simple form in this case -- namely, it degenerates at $E_2$.
\begin{theorem}[\cite{bbdg}, \cite{saito}]\label{decomp}
    Suppose $\pi:X\to B$ is a projective morphism between varieties, then there exists a (non-canonical) splitting
    \[\mathbf R\pi_\ast\operatorname{IC}_X=\bigoplus_{k=-r}^r{}^{\mathfrak p}\mathcal H^k(\mathbf R\pi_\ast\operatorname{IC}_X)[-k]\]
    in the category of mixed Hodge modules.
    Here, $r=r_\pi=\dim X\times_BX-\dim X$ is the defect of semismallness of $\pi$.
\end{theorem}
\begin{proposition}\label{decomp_TS}
    Suppose $\pi$ and $\pi^\phi$ are both projective.
    Then for each $k$, there are isomorphisms
    \[\Psi_k:{}^{\mathfrak p}\mathcal H^k(\mathbf R\pi_\ast\underline{\mathbb Z})\cong {}^{\mathfrak p}\mathcal H^k(\mathbf R\pi_\ast^\phi\underline{\mathbb Z})\]
    of perverse sheaves.
    Moreover, after restricting to each $U_i$, this isomorphism is the cohomology of
    \[\psi^i:\mathbf R\pi_\ast\underline{\mathbb Z}|_{U_i}\cong \mathbf R\pi_\ast^\phi\underline{\mathbb Z}|_{U_i}\]
    induced by the isomorphism $\pi^{-1}U_i\cong (\pi^\phi)^{-1}U_i$ relative to $U_i$.
    In particular, $\Psi_k$ is also an isomorphism of mixed Hodge modules. 
\end{proposition}
\begin{proof}
    Since perverse sheaves satisfy smooth descent (see \cite[\S3.7]{achar}), it suffices to show that the local isomorphisms ${}^{\mathfrak p}\mathcal H^k(\mathbf R\pi_\ast\underline{\mathbb Z}|_{U_i})\to {}^{\mathfrak p}\mathcal H^k(\mathbf R\pi_\ast^\phi\underline{\mathbb Z}|_{U_i})$ glue, meaning that the diagram
    \[\begin{tikzcd}
        {}^{\mathfrak p}\mathcal H^k(\mathbf R\pi_\ast\underline{\mathbb Z}|_{U_i})|_{U_i\cap U_j}\dar\rar&{}^{\mathfrak p}\mathcal H^k(\mathbf R\pi_\ast\underline{\mathbb Z}|_{U_j})|_{U_i\cap U_j}\dar\\
        {}^{\mathfrak p}\mathcal H^k(\mathbf R\pi_\ast^\phi\underline{\mathbb Z}|_{U_i})|_{U_i\cap U_j}\rar&{}^{\mathfrak p}\mathcal H^k(\mathbf R\pi_\ast^\phi\underline{\mathbb Z}|_{U_j})|_{U_i\cap U_j}
    \end{tikzcd}\]
    commutes.

    Denote by $\bar\pi:\pi^{-1}(U_i\cap U_j)=(\pi^\phi)^{-1}(U_i\cap U_j)\to U_i\cap U_j$ the common restriction of $\pi$ and $\pi^\phi$.
    The natural isomorphism $\underline{\mathbb Z}\to (\phi_{ij})_\ast\underline{\mathbb Z}$ induces an isomorphism ${}^{\mathfrak p}\mathcal H^k(\mathbf R\bar\pi_\ast\underline{\mathbb Z})\to {}^{\mathfrak p}\mathcal H^k(\mathbf R\bar\pi_\ast\underline{\mathbb Z})$.
    The commutativity of the diagram above amounts to the claim that this induced isomorphism is the identity.

    Write $\phi_{ij}=\exp(v)$ for some vertical vector field $v\in\Gamma(U_i\cap U_j,\pi_\ast T_{X/B})$.
    We then have a representation $\mathbb R\to\operatorname{End}({}^{\mathfrak p}\mathcal H^k(\mathbf R\bar\pi_\ast\underline{\mathbb Z}))$ sending $t$ to the map induced by $\exp(tv)$.
    The codomain being a finitely generated abelian group, the map must be constant.
    So $\phi_{ij}$ induces the same isomorphism as the identity.
\end{proof}
\begin{corollary}\label{isom_hodged_structure}
    Suppose $X$ and $X^\phi$ are both projective.
    Then there is an isomorphism $H^k(X;\mathbb Q)\cong H^k(X^\phi;\mathbb Q)$ of Hodge structures for each $k$.
\end{corollary}
\begin{proof}
    The preceding proposition shows that the factors in Theorem \ref{decomp} are the same for the fibration and its twist.
    Taking global cohomology shows the claim.
\end{proof}

\section{Cohomology rings of twists}
\subsection{Good twists}
We turn our focus to specific kinds of Tate-Shafarevich twists.
As before we are interested in a pair of data $(\pi:X\to B,\phi)$ consisting of a fibration $\pi:X\to B$ and a $1$-cocycle $\phi=(\phi_{ij})$ with values in $\underline{\operatorname{Aut}}_{X/B}^0$ over an open cover $\{U_i\}$ of $B$.
In this section we will only consider the case where $X$ is proper.
\begin{definition}
    A pair $(\pi:X\to B,\phi)$ is called \emph{good} if:

    (i) the morphism $\pi$ (hence $\pi^\phi$) is equidimensional;

    (ii) there exists relatively ample line bundles $L\in\operatorname{Pic}(X)$ and $L^\phi\in\operatorname{Pic}(X^\phi)$ such that $L|_{\pi^{-1}U_i}\cong L^\phi|_{(\pi^\phi)^{-1}U_i}$ for each $i$ (in particular, both $\pi$ and $\pi^\phi$ are projective);

    (iii) the perverse filtrations associated to $\pi$ and $\pi^\phi$ are multiplicative.
\end{definition}
Under the assumptions, we have $r_\pi=\dim X-\dim B$.

Let us explain the last condition.
Recall that given an equidimensional projective fibration $\pi:X\to B$, we get a natural map ${}^{\mathfrak p}\tau_{\le k}\mathbf R\pi_\ast\underline{\mathbb Q}\to \mathbf R\pi_\ast\underline{\mathbb Q}$ for each $k$, which induces an increasing filtration
\[P_kH^l(X;\mathbb Q)=\operatorname{Im}(H^l(B,{}^{\mathfrak p}\tau_{\le k+\dim B}\mathbf R\pi_\ast\underline{\mathbb Q})\to H^l(B,\mathbf R\pi_\ast\underline{\mathbb Q}))\]
on $H^l(B,\mathbf R\pi_\ast\underline{\mathbb Q})=H^l(X;\mathbb Q)$, known as the perverse filtration associated to $\pi$.
This filtration is said to be \emph{multiplicative} if
\[P_kH^l(X;\mathbb Q)\smile P_{k'}H^{l'}(X;\mathbb Q)\subset P_{k+k'}H^{l+l'}(X;\mathbb Q).\]

As a consequence of being multiplicative, the associated graded module of the perverse filtration on cohomology inherits a product structure from the cup product.
We therefore obtain a bigraded ring, with one grading given by the perverse filtration and the other given by the cohomology grading.
\begin{definition}
    Suppose the perverse filtration associated to a(n equidimensional projective) fibration $\pi:X\to B$ is multiplicative.
    The associated bigraded ring is defined by
    \[\mathbb H_\pi=\bigoplus_{k,l}\operatorname{Gr}_k^PH^l(X;\mathbb Q)\]
    with its product, which we will denote by $\smile^{\mathbb H}$, induced by the cup product.
\end{definition}
\begin{remark}
    It should be noted that the ring $\mathbb H_\pi$ may in general retain very little information about the cup product on $H^\ast(X;\mathbb Q)$.
    In particular, they are far from being isomorphic in general.
    Indeed, this can already be seen when the fibration is the projectivization of a vector bundle (over a smooth variety), as the Chern classes of the vector bundle can no longer be obtained from this associated graded ring.

    Such phenomenon also happens for generically abelian fibrations.
    In \cite{genbeauville}, the authors considered the compactified Jacobian over a family of integral curves with at worst simple nodes, and showed that the perverse filtration (which is still multiplicative in this case) does not admit a multiplicative splitting.

    In these kinds of contexts, the associated graded ring $\mathbb H_\pi$ usually behaves better than the cohomology ring $H^\ast(X;\mathbb Q)$.
    For example, in the case of universal compactified Jacobians over $\bar{\mathcal M}_{g,n}$ ($n\ge 1$), it was shown in \cite{bmsy25} that the associated graded ring does not depend on the degree nor the stability condition, even though the cohomology ring does.

    Nonetheless, there are situations where we can recover the cup product.
    See Theorem \ref{shenyinmain}.
\end{remark}
\subsection{The associated graded ring under a twist}\label{ss:Hpi_twist}
Let us examine the bigraded ring $\mathbb H$ under good twists.
Fix a good pair $(\pi:X\to B,\phi)$.

Let us write
\[\mathcal H_k={}^{\mathfrak p}\mathcal H^{k+\dim B}(\mathbf R\pi_\ast\underline{\mathbb Q})[-\dim B]\cong {}^{\mathfrak p}\mathcal H^{k+\dim B}(\mathbf R\pi_\ast^\phi\underline{\mathbb Q})[-\dim B]\]
where the identification is given by Proposition \ref{decomp_TS}.
Note that these are elements of $\operatorname{Perv}(B)[-\dim B]$.
The grading is chosen this way so that
\begin{equation}\label{eq}\tag{*}
    \mathbf R\pi_\ast\underline{\mathbb Q}=\bigoplus_{k=0}^{2r}\mathcal H_k[-k]=\mathbf R\pi^\phi_\ast\underline{\mathbb Q}
\end{equation}
where $r=\dim X-\dim B$.
So $\operatorname{Gr}_k^PH^l(X;\mathbb Q)$ is naturally identified with $H^{l-k}(\mathcal H_k)$.
Let us denote the induced isomorphism (of vector spaces) by $\Psi^{\mathbb H}:\mathbb H_\pi\to\mathbb H_{\pi^\phi}$.

\begin{theorem}[\cite{delignesplit}, \cite{fiveways}]
    A choice of a relatively ample line bundle $L$ induces a canonical splitting of the form in Theorem \ref{decomp}.
\end{theorem}
\begin{remark}
    The cited sources actually provide multiple different ways to produce canonical splittings associated to a relatively ample bundle.
    (A consistent choice of) any one of them would work for our purpose.
\end{remark}
Such a splitting induces maps which we will denote by
\[\operatorname{inj}_k^L:\mathcal H_k[-k]\to\mathbf R\pi_\ast\underline{\mathbb Q}, \operatorname{proj}^L_k:\mathbf R\pi_\ast\underline{\mathbb Q}\to\mathcal H_k[-k].\]

One consequence of the canonicity of the splitting is that these maps are compatible with restrictions to an open subset of $B$.
So the isomorphism $\Psi=\Psi_{L,L^\phi}$ of the form (\ref{eq}) induced by the choice of $L,L^\phi$ does restrict to $\Psi^i$ (more precisely, $\Psi^i_{\mathbb Q}$) as described in Proposition \ref{decomp_TS} on each $U_i$.

In particular, the restriction of the diagram
\[
\begin{tikzcd}
    \mathbf R\pi_\ast\underline{\mathbb Q}\otimes\mathbf R\pi_\ast\underline{\mathbb Q}\rar\dar{\Psi\otimes\Psi}& \mathbf R\pi_\ast\underline{\mathbb Q}\dar{\Psi}\\
    \mathbf R\pi_\ast^\phi\underline{\mathbb Q}\otimes\mathbf R\pi_\ast^\phi\underline{\mathbb Q}\rar&\mathbf R\pi_\ast^\phi\underline{\mathbb Q}
\end{tikzcd}
\]
to each $U_i$ also commutes.
Note that we cannot conclude that the diagram commutes globally:
Maps between complexes are generally not determined locally.

Now, for each $k,k'$, we have a map
\[\begin{tikzcd}
    \mathcal H_k[-k]\otimes\mathcal H_{k'}[-k']\rar{\operatorname{inj}_k^L\otimes \operatorname{inj}_{k'}^L}&\mathbf R\pi_\ast\underline{\mathbb Q}\otimes\mathbf R\pi_\ast\underline{\mathbb Q}\rar&\mathbf R\pi_\ast\underline{\mathbb Q}\rar{\operatorname{proj}_{k+k'}^L}&\mathcal H_{k+k'}[-k-k']
\end{tikzcd}\]
which correspond to a map $\mathcal H_k\otimes\mathcal H_{k'}\to \mathcal H_{k+k'}$.
Let us call it $\operatorname{prod}_{k,k'}^L$.

Taking global cohomology on $\operatorname{prod}_{k,k'}^L$ gives a map
\[\mathbf R\Gamma(\mathcal H_k)\otimes \mathbf R\Gamma(\mathcal H_{k'})\to\mathbf R\Gamma(\mathcal H_k\otimes\mathcal H_{k'})\to\mathbf R\Gamma(\mathcal H_{k+k'})\]
which then induces a binary operation
\[b:H^{l-k}(\mathcal H_k)\otimes H^{l'-k'}(\mathcal H_{k'})\to H^{l+l'-k-k'}(\mathcal H_{k+k'})\]
for each $k,k',l,l'$.
\begin{proposition}
    The operation $b$ (restricted to pure tensors) coincides with the induced cup product $\operatorname{Gr}^P_kH^l(X;\mathbb Q)\times\operatorname{Gr}_{k'}^PH^{l'}(X;\mathbb Q)\to\operatorname{Gr}_{k+k'}^PH^{l+l'}(X;\mathbb Q)$.
\end{proposition}
\begin{proof}
    Let $\mathcal Q_k=\mathcal H_0\oplus\cdots\oplus\mathcal H_k[-k]={}^{\mathfrak p}\tau_{\le k+\dim B}(\mathbf R\pi_\ast\underline{\mathbb Q})$.
    In this case, the natural map associated to the perverse truncation can be written as
    \[\operatorname{filt}_k:\operatorname{inj}_0^L\oplus\cdots\oplus\operatorname{inj}_k^L:\mathcal Q_k\to\mathbf R\pi_\ast\mathbb Q.\]
    
    Since $\operatorname{inj}^L$ and $\operatorname{proj}^L$ both come from a splitting, $\operatorname{filt}_k$ has a retraction given by
    \[\operatorname{ret}_k^L=\operatorname{proj}_0^L\oplus\cdots\oplus\operatorname{proj}_k^L:\mathbf R\pi_\ast\mathbb Q\to \mathcal Q_k\]
    which depends on $L$.

    Consider now the map $\operatorname{fprod}_{k,k'}$ given by the composition
    \[\begin{tikzcd}
        \mathcal Q_k\otimes\mathcal Q_{k'}\rar{\operatorname{filt}_k^L\otimes \operatorname{filt}_{k'}^L}&\mathbf R\pi_\ast\underline{\mathbb Q}\otimes\mathbf R\pi_\ast\underline{\mathbb Q}\rar&\mathbf R\pi_\ast\underline{\mathbb Q}.
    \end{tikzcd}\]
    This map does not necessarily factor through $\operatorname{filt}_{k+k'}:\mathcal Q_{k+k'}\to\mathbf R\pi_\ast\underline{\mathbb Q}$.
    However, if one were to pass this to cohomology, then
    \[\mathbf R\Gamma(\operatorname{filt}_k):\mathbf R\Gamma(\mathcal Q_k)\to \mathbf R\Gamma(\mathbf R\pi_\ast\underline{\mathbb Q})\]
    is exactly the inclusion of the filtered piece $P_kH^\ast(X;\mathbb Q)\subset H^\ast(X;\mathbb Q)$.
    Since we have assumed that the perverse filtration is multiplicative, the cohomological realization can then be factored as
    \[\begin{tikzcd}
        \mathbf R\Gamma(\mathcal Q_k)\otimes\mathbf R\Gamma(\mathcal Q_{k'})\rar\arrow[swap,dashed]{drr}{\operatorname{cup}_{k,k'}}&\mathbf R\Gamma(\mathbf R\pi_\ast\underline{\mathbb Q})\otimes\mathbf R\Gamma(\mathbf R\pi_\ast\underline{\mathbb Q})\rar&\mathbf R\Gamma(\mathbf R\pi_\ast\underline{\mathbb Q})\\
        &&\mathbf R\Gamma(\mathcal Q_{k+k'})\uar[swap]{\mathbf R\Gamma(\operatorname{filt}_{k+k'})}
    \end{tikzcd}\]
    
    Now since $\operatorname{ret}_{k+k'}$ is a retraction of $\operatorname{filt}_{k+k'}$, we know that $\mathbf R\Gamma(\operatorname{ret}_{k+k'})$ is a retraction of $\mathbf R\Gamma(\operatorname{filt}_{k+k'})$, so
    \begin{align*}
        \operatorname{cup}_{k,k'}&=\mathbf R\Gamma(\operatorname{ret}_{k+k'})\circ\mathbf R\Gamma(\operatorname{fprod}_{k,k'})=\mathbf R\Gamma(\operatorname{ret}_{k+k'}\circ\operatorname{fprod}_{k,k'})
    \end{align*}
    which shows the claim since we can naturally identify $\mathcal Q_\ast/\mathcal Q_{\ast-1}=\mathcal H_\ast$.
\end{proof}
Doing the same thing for the twisted fibration $\pi^\phi$ gives another map $\operatorname{prod}_{k,k'}^{L^\phi}$.
The key observation here is the following:
\begin{proposition}
    We have $\operatorname{prod}_{k,k'}^L=\operatorname{prod}_{k,k'}^{L^\phi}$ given that at least one of $\mathcal H_k,\mathcal H_{k'}$, and $\mathcal H_{k+k'}$ is a local system.
\end{proposition}
\begin{proof}
    We have already seen that the two maps are the same after restricting to $U_i$ for each $i$.
    As such, it suffices to show that any map $\mathcal H_k\otimes\mathcal H_{k'}\to \mathcal H_{k+k'}$ canonically corresponds to a map between perverse sheaves, for then it is determined locally.
    This is clear if either one of $\mathcal H_k,\mathcal H_{k'}$ is a local system.
    If $\mathcal H_{k+k'}$ is a local system, we identify:
    \begin{align*}
        \operatorname{Hom}(\mathcal H_k\otimes\mathcal H_{k'}, \mathcal H_{k+k'})&=\operatorname{Hom}(\mathcal H_{k'},\mathbf R\underline{\operatorname{Hom}}(\mathcal H_k,\mathcal H_{k+k'}))
    \end{align*}
    The claim would follow if $\mathbf R\underline{\operatorname{Hom}}(\mathcal H_k,\mathcal H_{k+k'})\in\operatorname{Perv}(B)[-\dim B]$.
    Without loss of generality, we may assume $\mathcal H_{k+k'}=\underline{\mathbb Q}$.
    Since $B$ is the base of an equidimensional fibration, we have $\operatorname{IC}_B[-\dim B]\cong\underline{\mathbb Q}$.
    Therefore
    \[\mathbf R\underline{\operatorname{Hom}}(\mathcal H_k,\mathcal H_{k+k'})=(\mathbb D\mathcal H_k)[-2\dim B]\in\operatorname{Perv}(B)[-\dim B].\qedhere\]
\end{proof}
Now fix an ample class $\kappa$ on $B$.
We may view it as an element of $H^2(\mathcal H_0)$ which we have identified with $\operatorname{Gr}^P_0H^2(X;\mathbb Q)$ as well as $\operatorname{Gr}^P_0H^2(X^\phi;\mathbb Q)$.
\begin{corollary}\label{intersectionform}
    The isomorphism of vector spaces $\Psi^{\mathbb H}:\mathbb H_\pi\to\mathbb H_{\pi^\phi}$ preserves:

    (a) the endomorphism $\kappa\smile^{\mathbb H}-$, and

    (b) the intersection pairing
    \[\langle\alpha,\beta\rangle_{\mathbb H}=\int\alpha\smile^{\mathbb H}\beta\]
    where the integral is the projection to $\operatorname{Gr}^P_{2r}H^{2\dim X}$.
\end{corollary}
\begin{proof}
    Choosing $k=0$ (resp.~$k+k'=2r$) in the preceding proposition yields (a) (resp.~(b)).
\end{proof}
\subsection{The relative Lefschetz action}\label{ss:lefschetz}
There is another endomorphism that is preserved by $\Psi^{\mathbb H}$.
The relatively ample line bundle $L$ induces a map $\mathbf R\pi_\ast\underline{\mathbb Q}\to \mathbf R\pi_\ast\underline{\mathbb Q}[2]$ which gives a map $\eta:\mathcal H_0\to\mathcal H_2$.
Denote by $\mathbb L_k$ the composition
\[\begin{tikzcd}
    \mathcal H_k=\mathcal H_k\otimes\mathcal H_0\rar{\operatorname{id}\otimes\eta}&\mathcal H_k\otimes\mathcal H_2\rar{\operatorname{prod}_{k,2}}&\mathcal H_{k+2}.
\end{tikzcd}\]

Similarly, we get $\eta^\phi$ and $\mathbb L_k^\phi$ by doing this for $L^\phi$.

\begin{lemma}
    We have $\mathbb L_\bullet=\mathbb L_\bullet^\phi$.
\end{lemma}
\begin{proof}
    These are maps between objects of $\operatorname{Perv}(B)[-\dim B]$, so one can check their coincidence locally.
    The latter follows from the fact that $L|_{\pi^{-1}U_i},L^\phi|_{(\pi^\phi)^{-1}U_i}$ are isomorphic for each $i$.
\end{proof}
If one were to pass $\mathbb L_\bullet$ to cohomology, one would get an endomorphism of $\mathbb H_\pi$.
\begin{proposition}\label{lef_action}
    (a) The endomorphism that $\mathbb L_\bullet$ defines on $\mathbb H_\pi$ is $\eta\smile^{\mathbb H}-$, where $\eta$ is viewed as a class in $\operatorname{Gr}^P_2H^2(X;\mathbb Q)$ via
    \[\operatorname{Hom}(\mathcal H_0,\mathcal H_2)=H^0(\mathcal H_2)=\operatorname{Gr}^P_2H^2(X;\mathbb Q).\]

    (b) The isomorphism $\Psi^{\mathbb H}$ conjugates the operators $\eta\smile^{\mathbb H}-$ and $\eta^\phi\smile^{\mathbb H}-$.
\end{proposition}
\begin{proof}
    These already follow from the existing discussions.
\end{proof}
\section{Lagrangian fibrations}
\subsection{Recollections on Lagrangian fibrations}\label{ss:rec_lagfib}
Let us now specialize to the situation where $X$ is an irreducible hyper-K\"ahler variety.
\begin{definition}
    A compact K\"ahler manifold $X$ is irreducible hyper-K\"ahler if it is holomorphic symplectic and
    \[h^1(X)=0,h^{2,0}(X)=1.\]
    Fibrations whose total space is irreducible hyper-K\"ahler are called Lagrangian fibrations.
\end{definition}
\begin{remark}
    Any irreducible hyper-K\"ahler manifold as defined above is simply connected \cite{schwald}.
    So our definition here matches with the classical definition \cite{huy}.
\end{remark}
For the rest of this article, we fix a projective Lagrangian fibration $\pi:X\to B$ and a torsion class $\phi\in\Sha_{X/B}$.
Torsion Tate-Shafarevich twists usually coincide with Tate-Shafarevich twists in the \'etale topology (see \S\ref{ss:torsion}).

Let us show that $(\pi:X\to B,\phi)$ is a good pair.

Write $m$ for the order of $\phi$.
As observed in \cite[Lemma 3.0.2]{abashevaii}, one can choose $\phi$ so that $m$ kills the local cocycles already.
Based on this observation, it was proved that:
\begin{theorem}[{\cite[\S3]{abashevaii}}]\label{torsion_rel_ample}
    For any line bundle $L$ on $X$, possibly after replacing $L$ by a positive power, there exists a $1$-cocycle of isomorphisms
    \[f_{ij}:L|_{U_i\cap U_j}\to(\phi_{ij})_\ast L|_{U_i\cap U_j},\]
    unique up to $\mathcal O_{U_i\cap U_j}^\times$.
    In particular, they glue to a line bundle $L^\phi$ on $X^\phi$, well-defined up to $\operatorname{Pic}(B)$.

    This establishes an isomorphism
    \[\operatorname{Pic}(X)_{\mathbb Q}/\pi^\ast\operatorname{Pic}(B)_{\mathbb Q}\cong\operatorname{Pic}(X^\phi)_{\mathbb Q}/(\pi^\phi)^\ast\operatorname{Pic}(B)_{\mathbb Q}\]
    sending $\pi$-ample line bundles to $\pi^\phi$-ample line bundles.
\end{theorem}
\begin{remark}
    The ambiguity of $\operatorname{Pic}(B)_{\mathbb Q}$ can be removed:
    Call $f_{ij}^+$ the composition
    \[L|_{U_i\cap U_j}\to(\phi_{ij})_\ast L|_{U_i\cap U_j}\to\cdots\to (\phi_{ij}^m)_\ast L|_{U_i\cap U_j}=L_{U_i\cap U_j}\]
    which is an element of $\mathcal O_{U_i\cap U_j}^\times$.
    Then the maps
    \[\tilde{f}_{ij}=f_{ij}^m(f_{ij}^+)^{-1}:L^m|_{U_i\cap U_j}\to (\phi_{ij})_\ast L^m|_{U_i\cap U_j}\]
    satisfy $\tilde{f}_{ij}^+=1$.
    So there is always a choice of $f_{ij}$ with $f_{ij}^+=1$, possibly after replacing $L$ by a positive power of it.
    This additional requirement allows an isomorphism $\operatorname{Pic}(X)_{\mathbb Q}\cong\operatorname{Pic}(X^\phi)_{\mathbb Q}$.
\end{remark}
\begin{corollary}
    The twist $X^\phi$ is an irreducible hyper-K\"ahler variety, and one can find relatively ample $L\in\operatorname{Pic}(X),L^\phi\in\operatorname{Pic}(X^\phi)$ that are such that
    \[L|_{\pi^{-1}U_i}\cong L^\phi|_{(\pi^\phi)^{-1}U_i}\]
    for each $i$.
\end{corollary}
\begin{proof}
    It is holomorphic symplectic by \cite[Corollary 3.7]{AR21} and the cohomological computations can be found in \cite[\S5]{abashevaii}.
\end{proof}
Lagrangian fibrations are equidimensional by \cite{matequid}.
One particularly nice property of Lagrangian fibrations is that the perverse filtration associated to it is as nice as possible as far as the cup product is concerned.
\begin{theorem}[{\cite[Appendix A]{shenyin}}]\label{shenyinmain}
    The perverse filtration associated to a Lagrangian fibration is multiplicative.
    Moreover, there is an isomorphism of graded rings $H^\ast(X;\mathbb Q)\cong\mathbb H_\pi$ (where only the cohomology grading is considered on the latter).
\end{theorem}
Hence $(\pi,\phi)$ is a good pair, and the associated graded ring is isomorphic to the cohomology ring.
\subsection{Topological invariants under twists}\label{ss:specialize_lag}
The results we obtained in the previous section specialize, in the case of a Lagrangian fibration, to the invariance of certain topological information of the total space under a torsion twist.
\begin{proposition}\label{HRpairing_proved}
    There exist ample classes $\omega,\omega^\phi$ on $X,X^\phi$ such that the isomorphism $H^\ast(X;\mathbb Q)\cong H^\ast(X^\phi;\mathbb Q)$ as in Corollary \ref{isom_hodged_structure} may be chosen so that:

    (a) it conjugates the endomorphisms $\omega\smile-$ and $\omega^\phi\smile-$, and

    (b) it preserves the intersection pairing, \emph{i.e.}~the bilinear form given by
    \[\langle\alpha,\beta\rangle:=\int\alpha\smile\beta.\]
\end{proposition}
\begin{proof}
    Since $q_X(c_1(L),\kappa)\neq 0$, we can choose $\lambda\in\mathbb Q$ such that $q_X(c_1(L)+\lambda\kappa)=0$.
    Similarly, we can choose $\lambda^\phi\in\mathbb Q$ such that $q_{X^\phi}(c_1(L^\phi)+\lambda^\phi\kappa)=0$.
    This means that we can assume, without loss of generality, that our choice of relatively ample line bundles $L,L^\phi$ satisfies
    \[q_X(c_1(L))=0=q_{X^\phi}(c_1(L^\phi)).\]

    By the proof of the main theorem in \cite{shenyin}, this means that the splitting in Theorem \ref{shenyinmain} can be chosen so that the image of $c_1(L)$ in $\mathbb H$ generates exactly $\operatorname{Gr}^P_2H^2(X;\mathbb Q)$.
    Of course, this image is exactly the class $\eta$ described in Proposition \ref{lef_action}.

    Moreover, the $\kappa$ action on $H^\ast(X;\mathbb Q)$ coincides with the $\kappa$ action on $\mathbb H_\pi$.
    So for any $N$, the class $c_1(L)+N\kappa\in H^\ast(X;\mathbb Q)$ is sent to $\eta+N\kappa\in\mathbb H_\pi$ under an isomorphism of the form in Theorem \ref{shenyinmain}.

    All of these are true on the $X^\phi$ side as well.
    Hence we can just choose $N$ such that both $\omega=c_1(L)+N\kappa$ and $\omega^\phi=c_1(L^\phi)+N\kappa$ are ample.
    Both claims then follow from their analogues on the associated graded ring, which are true by Corollary \ref{intersectionform} and Proposition \ref{lef_action}.
\end{proof}
Given a K\"ahler class $\omega$ on a compact manifold $X$, we define the Hodge-Riemann form on $H^\ast(X;\mathbb Q)$ by setting
\[\langle\alpha,\beta\rangle_\omega:=\int\left(1+\omega+\omega^2+\cdots\right)\smile\alpha\smile\beta.\]
\begin{corollary}
    Under the isomorphism in the preceding proposition, the Hodge-Riemann form is preserved.
\end{corollary}
\subsection{Beauville-Bogomolov-Fujiki lattice}\label{ss:bbf}
Write $2n$ for the dimension of $X$ (over $\mathbb C$).
For any ample class $\omega\in H^2(X;\mathbb Q)$, we define a quadratic form on $H^2(X;\mathbb Q)$ by the formula
\[q^\omega(\alpha):=(2n-1)\langle 1,1\rangle_\omega\langle\alpha,\alpha\rangle_\omega-(2n-2)\langle 1,\alpha\rangle_\omega^2\]
for any $\alpha\in H^2(X;\mathbb Q)$.
From our previous discussions, we immediately get:
\begin{corollary}\label{hodge_similitude}
    Let $\omega,\omega^\phi$ be as in Proposition \ref{HRpairing_proved}.
    Then there is a Hodge isometry
    \[(H^2(X;\mathbb Q),q^\omega)\cong (H^2(X^\phi;\mathbb Q),q^{\omega^\phi}).\]
\end{corollary}
\begin{theorem}[\cite{beauville}, \cite{fujiki}, \cite{verbitskythesis}, \cite{bogomolov}, \cite{huy}]
    The quadratic form $q^\omega$ is non-degenerate of signature $(3,b_2(X)-3)$.

    If $\omega'$ is a different ample class, then $q^\omega=\lambda q^{\omega'}$ for some $\lambda\in\mathbb Q^+$.
    Let $q_X=\lambda q^\omega$ be such that $\lambda\in\mathbb Q^+$ and $q_X$ is integral and indivisible on $H^2(X;\mathbb Z)$.
    Then there is a constant $c_X\in\mathbb Q^+$ such that
    \[\int\alpha^{2n}=c_X\frac{(2n)!}{2^nn!}q_X(\alpha)^n\]
    for every $\alpha\in H^2(X;\mathbb Z)$.
\end{theorem}
\begin{definition}
    The form $q_X$ is called the Beauville-Bogomolov-Fujiki (BBF) form, and $c_X$ is called the Fujiki constant.
\end{definition}
Recall that $H^2(X;\mathbb Z)$ is torsion-free since $X$ is simply connected.
When equipped with $q_X$, it is called the BBF lattice of $X$.
\begin{example}\label{example:BMsystem}
    Let us check Corollary \ref{hodge_similitude} for certain fibrations of the type considered in Example \ref{example:compactified_jac_twist}.
    Take a K3 surface $S$ with $\operatorname{Pic}(S)=\mathbb ZH$ for some ample class $H$.
    Then consider the moduli space $M_\chi=M(0,H,\chi)$ of stable sheaves with Mukai vector $v_\chi=(0,H,\chi)\in \tilde{H}(S;\mathbb Z)$, which fibers over $|H|$.
    This is a compactified Jacobian fibration of the ample linear system, with different values of $\chi$ corresponding to different values of degree.

    The BBF lattice $H^2(M_\chi;\mathbb Z)$ in this case is isomorphic to $v_\chi^\perp\subset\tilde{H}(S;\mathbb Z)$.
    But for any $\chi,\chi'$ nonzero, there is a Hodge isometry of $\tilde{H}(S;\mathbb Q)$ with itself given by
    \[(a,\ell,b)\mapsto (a\chi/\chi',\ell,b\chi'/\chi),\]
    and it takes $v_\chi$ to $v_{\chi'}$.
    Hence $H^2(M_\chi;\mathbb Q)$ and $H^2(M_{\chi'};\mathbb Q)$ are indeed Hodge isometric.
    To get to $\chi=0$, simply observe that $M_\chi\cong M_{\chi+H^2}$.
\end{example}
Generally, it can be impossible to find an isomorphism $H^2(X;\mathbb Q)\cong H^2(X^\phi;\mathbb Q)$ coming from a splitting of the perverse filtrations that preserves both a multiple of the respective BBF forms and the integral cohomology, since any such isomorphism is Hodge.
\begin{example}
    Recall in Example \ref{ex:elliptick3} that there are non-birational K3 surfaces related by a Tate-Shafarevich twist.
    The corresponding Tate-Shafarevich class is torsion, for example by Theorem \ref{projectiveimpliestorsion}.
    So the integral BBF lattices cannot be Hodge isometric by the global Torelli theorem \cite{k3torelli}.
\end{example}
So the present method does not seem to reach the full extent of the integral structure.
Let us briefly discuss the limited information that one can get.

As before, write $\kappa\in H^2(X;\mathbb Q)$ for the pullback of an ample class on $B$.
\begin{definition}
    The local BBF lattice for the Lagrangian fibration $\pi:X\to B$ is the sublattice $\Lambda_\pi=\kappa^\perp\cap H^2(X;\mathbb Z)$ equipped with the restriction of the BBF form.
    The reduced local BBF lattice $\Lambda_\pi'$ has the same underlying $\mathbb Z$-module as $\Lambda_\pi$, but is equipped with the form $\lambda_\pi^{-1} q_X$ with $\lambda_\pi\in\mathbb Z^+$ chosen so that $\Lambda_\pi'$ is integral and indivisible.
\end{definition}
\begin{lemma}
    We have $\mathbb Q\kappa=P_0H^2(X;\mathbb Q)$ and $\kappa^\perp=P_1H^2(X;\mathbb Q)$.
\end{lemma}
\begin{proof}
    The first statement is clear.
    As for the second statement, note first that $q_X(\kappa)=0$.
    Also, from Theorem \ref{shenyinmain} we know that the perverse filtration associated to a Lagrangian fibration is multiplicative, so for any $\alpha\in P_1H^2(X;\mathbb Q)$, we have
    \[c_X\frac{(2n)!}{2^nn!}q_X(\alpha+\beta)=\int(\alpha+\beta)^{2n}=\int\alpha^{2n}=c_X\frac{(2n)!}{2^nn!}q_X(\alpha),\]
    hence $q_X(\alpha,\beta)=0$.
    Therefore $P_1H^2(X;\mathbb Q)\subset\kappa^\perp$.

    On the other hand, if the inclusion were strict, then since $P_1H^2(X;\mathbb Q)$ has codimension $1$ in $H^2(X;\mathbb Q)$ we must have $\kappa^\perp=H^2(X;\mathbb Q)$.
    This however contradicts the nondegeneracy of $q_X$.
\end{proof}
In particular, $\Lambda_\pi$ has corank $1$ in $H^2(X;\mathbb Z)$.
Now let
\[\mathcal H_k^{\mathbb Z}={}^{\mathfrak p}\mathcal H^k(\mathbf R\pi_\ast\underline{\mathbb Z})[-\dim B]={}^{\mathfrak p}\mathcal H^k(\mathbf R\pi_\ast^\phi\underline{\mathbb Z})[-\dim B]\]
with the identification given by Proposition \ref{decomp_TS}.
Similar to the argument of \cite[Lemma 3.8.3]{achar}, we have $\mathcal H_k=0$ if $k<0$.
\begin{proposition}
    Suppose $H^3(\mathcal H_0^{\mathbb Z})$ and $H^0(\mathcal H_2^{\mathbb Z})$ are torsion-free.
    Then there is a Hodge isometry $\Lambda_\pi'\cong\Lambda_{\pi^\phi}'$, and we have $\lambda_\pi^nc_X=\lambda_{\pi^\phi}^nc_{X^\phi}$.
\end{proposition}
\begin{proof}
    Consider the perverse Leray sequence for perverse sheaves over $\mathbb Z$ which takes the form
    \[E_2^{p,q}=H^p(\mathcal H_q^{\mathbb Z})\Rightarrow H^{p+q}(X;\mathbb Z).\]

    After base-change to $\mathbb Q$, this recovers the corresponding Leray sequence for $\mathbb Q$ coefficients, which we know from Theorem \ref{decomp} is degenerate at $E_2$.
    Hence all differentials in $E_2^{p,q}$ have torsion images.

    The spectral sequence converges to a filtration of $H^2(X;\mathbb Z)$.
    Let us denote it by $P_kH^2(X;\mathbb Z)$ for $k\le 2$ which, after extending the coefficient to $\mathbb Q$, recovers the perverse filtration on $H^2(X;\mathbb Q)$.

    Since $H^0(\mathcal H_2^{\mathbb Z})$ is torsion-free, $\operatorname{Gr}^P_2H^2(X;\mathbb Z)$ is torsion-free and hence isomorphic to $\mathbb Z$.
    Consequently, we have
    \[P_1H^2(X;\mathbb Z)=P_1H^2(X;\mathbb Q)\cap H^2(X;\mathbb Z)=\kappa^\perp\cap H^2(X;\mathbb Z)\]
    by the preceding lemma.

    As in the proof of the preceding lemma, $q_X(\kappa,-)$ is zero on $P_1H^2(X;\mathbb Q)$, hence the form $q^\omega$ descends to a quadratic form $\bar{q}^\omega$ on $\operatorname{Gr}^P_1H^2(X;\mathbb Q)$.

    On the other hand, the isomorphism in Corollary \ref{hodge_similitude} respects the perverse filtration by construction.
    The isomorphism on the graded piece $\operatorname{Gr}_1^PH^2(X;\mathbb Q)\cong\operatorname{Gr}_1^PH^2(X^\phi;\mathbb Q)$ comes from extending the isomorphism in Proposition \ref{decomp_TS} to $\mathbb Q$.
    In particular, it preserves $H^1(\mathcal H_1^{\mathbb Z})^{\rm tf}$.

    Since $H^3(\mathcal H_0^{\mathbb Z})$ is torsion-free, it must be zero.
    Therefore $\operatorname{Gr}^P_1H^2(X;\mathbb Z)^{\rm tf}=H^1(\mathcal H_1^{\mathbb Z})^{\rm tf}$.
    We thus have a Hodge isometry
    \[(\operatorname{Gr}^P_1H^2(X;\mathbb Z)^{\rm tf},\bar{q}^\omega)\cong(\operatorname{Gr}^P_1H^2(X^\phi;\mathbb Z)^{\rm tf},\bar{q}^{\omega^\phi}).\]

    Now observe that there is a Hodge isometry $(\operatorname{Gr}^P_1H^2(X;\mathbb Z)^{\rm tf},\bar{q}^\omega)\oplus\mathbb Z(0)\cong (P_1H^2(X;\mathbb Z),q^\omega)$, where $\mathbb Z(0)$ is the rank $1$ lattice with zero pairing.
    Indeed, the $\mathbb Z(0)$ factor is precisely given by $\mathbb Q\kappa\cap H^2(X;\mathbb Z)$.
    So we actually have a Hodge isometry
    \[(P_1H^2(X;\mathbb Z),q^\omega)\cong(P_1H^2(X^\phi;\mathbb Z),q^{\omega^\phi}).\]

    There is then a constant $\mu\in\mathbb Q^+$ such that $\mu q^\omega$ and $\mu q^{\omega^\phi}$ are both integral and indivisible on $P_1H^2(X;\mathbb Z)$ and $P_1H^2(X^\phi;\mathbb Z)$, respectively.
    This induces a Hodge isometry $\Lambda_\pi'\cong\Lambda_{\pi^\phi}'$.

    Moreover, we know that $\lambda_\pi\mu q^\omega=q_X$.
    So
    \[\langle 1,1\rangle_\omega=c_X\frac{(2n)!}{2^nn!}q_X(\omega,\omega)^n=c_X\frac{(2n)!}{2^nn!}\lambda_\pi^n\mu^nq^\omega(\omega,\omega)^n.\]
    But we also have $q^\omega(\omega,\omega)=\langle 1,1\rangle_\omega=\langle 1,1\rangle_{\omega^\phi}$.
    So $\lambda_\pi^nc_X=\lambda_{\pi^\phi}^nc_{X^\phi}$.
\end{proof}
\begin{remark}
    By \cite{K32}, a hyper-K\"ahler variety $X$ has type K3${}^{[2]}$ if and only if $c_X=1$ and there exists classes $\ell,m\in H^2(X;\mathbb Z)$ with $q_X(\ell)=0$ and $q_X(\ell,m)=1$.

    Suppose that our Lagrangian fibration $\pi:X\to B$ satisfies the torsion-free conditions in the preceding proposition.
    If $X$ is of K3$^{[2]}$ type and the aforementioned classes $\ell,m$ can already be found in $\Lambda_\pi$, then $X^\phi$ must also be of K3${}^{[2]}$-type.

    Indeed, it suffices to show that $\lambda_{\pi^\phi}=c_{X^\phi}=1$.
    We have $\lambda_\pi=c_X=1$ by assumptions, so $\lambda_{\pi^\phi}^nc_{X^\phi}=1$.
    Thus
    \[3\left(\frac{q_{X^\phi}(\alpha)}{\lambda_{\pi^\phi}}\right)^2=\int\alpha^4\in\mathbb Z\]
    for any $\alpha\in H^2(X^\phi;\mathbb Z)$.
    Since the BBF lattice is indivisible by definition, the greatest common factor of all $q_{X^\phi}(\alpha)$ is $2$.
    Therefore $\lambda_{\pi^\phi}=1$, hence $c_{X^\phi}=1$.
\end{remark}

\section{Fibrations with \texorpdfstring{$C^\infty$}{Cinfty}-sections}
\subsection{The identity component of \texorpdfstring{$\Sha$}{Sha}}
From now on, we assume that the base $B$ is smooth.
This forces $B$ to be isomorphic to a projective space.

There is one part of the group $\Sha$ where twists by its elements are already a deformation of the original fibration.
\begin{definition}
    We denote by $\Sha^0_{X/B}\subset\Sha_{X/B}$ the image of the map
    \[H^1(\exp):H^1(B,\pi_\ast T_{X/B})\to H^1(B,\underline{\operatorname{Aut}}^0_{X/B})=\Sha.\]
\end{definition}
\begin{theorem}[{\cite[Corollary 2.6, Proposition 3.3]{AR21}}]
    There is a holomorphic family (the ``Tate-Shafarevich family'') of fibrations over $H^1(B,\pi_\ast T_{X/B})$, such that its fiber over $t\in H^1(B,\pi_\ast T_{X/B})$ is the twisted fibration $\pi^\phi:X^\phi\to B$ where $\phi=H^1(\exp)(t)$.
\end{theorem}
This holds true without the Lagrangian assumption.
In the Lagrangian case, this is a one-dimensional family:
We have an isomorphism
\[H^1(B,\pi_\ast T_{X/B})\cong H^{1,1}(B)\cong\mathbb C\]
by \cite[Corollary 2.6]{AR21}.

There is another natural $\mathbb C$-family of Lagrangian fibrations starting from a given one (satisfying our stated assumptions), called a degenerate twistor deformation.

We will not need a technical definition of degenerate twistor deformations.
Interested reader may confer \cite{degtwistor}.
Roughly, given a choice of a holomorphic symplectic form $\sigma$ on $X$ and $t\in\mathbb C$, one can associate an integrable complex structure $I_t$ with the $\mathbb C$-valued $2$-form $\sigma_t=\sigma+t\pi^\ast\kappa$ (where $\kappa$ is the class of a hyperplane in $B$), so that $\sigma_t$ is holomorphic symplectic with respect to $I_t$.
\begin{theorem}[\cite{degtwistor}]
    The map $\pi:(X,I_t)\to B$ is holomorphic and Lagrangian with respect to $\sigma_t$ for all $t\in\mathbb C$.
    And the biholomorphism classes of its fibers are independent of $t$.

    Moreover, one can put an integrable complex $I_{\rm tw}$ structure on $X\times\mathbb C$ such that:

    (i) the projection to $\mathbb C$ is holomorphic;

    (ii) the restriction of the complex structure to the fiber over $t\in\mathbb C$ is $I_t$.
\end{theorem}
The family $(X\times\mathbb C,I_{\rm tw})$ is called the degenerate twistor family.
The complex manifolds in this family are all K\"ahler (\cite{kahlertwistor}).

Since we have chosen $\kappa$ to be a hyperplane class on $B$, the base for the degenerate twistor family can be naturally identified with $H^{1,1}(B)$.
\begin{theorem}[\cite{AR21}]\label{abashevarogov}
    Under the isomorphism $H^1(B,\pi_\ast T_{X/B})\cong H^{1,1}(B)$, the degenerate twistor family is isomorphic to the Tate-Shafarevich family.
\end{theorem}
\subsection{Torsions in \texorpdfstring{$\Sha$}{Sha} and algebraic twists}\label{ss:torsion}
Let us discuss the role of torsion elements in $\Sha$.
We have already mentioned Theorem \ref{torsion_rel_ample} that says, among other things, that twists of projective Lagrangian fibrations by torsion elements are projective.
There is a partial converse to this.
\begin{theorem}[{\cite[Theorem A]{abashevaii}}]\label{projectiveimpliestorsion}
    Let $\Sha'\subset\Sha$ be the subgroup consisting of elements whose image in $\Sha/\Sha^0$ is torsion.
    Then for $\phi\in\Sha'$, $X^\phi$ is projective if and only if $\phi$ is torsion.
\end{theorem}
So torsion points should have some kind of interpretation within an algebraic context.
For a start, one needs to make sense of what an algebraic twist actually is.

\begin{theorem}[\cite{kim_neron}; see also \cite{sacca}]\label{group_scheme_action}
    Suppose fibers of the Lagrangian fibration are non-multiple, \emph{i.e.}~they all have at least one reduced component.
    Then there is a connected group scheme $P=P_\pi\to B$ whose analytification represents $\underline{\operatorname{Aut}}^0_{X/B}$.
\end{theorem}
Note that this is $P^{\circ\circ}$ in the notation of \cite{kim_neron}.

This group scheme allows one to make sense of Tate-Shafarevich twists in the \'etale topology.
The construction is the same as before, except that instead of performing a holomorphic gluing in the analytic topology, we apply \'etale descent to the local data.
This gives us an algebraic space fibered over $B$.
\begin{definition}
    The \'etale Tate-Shafarevich group for the fibration is $\Sha_{\text{\'et}}=H_{\text{\'et}}^1(B,P)$.
\end{definition}
One immediately gets a natural map (given by analytification) $\Sha_{\text{\'et}}\to\Sha$, and an \'etale twist by an element in the former group is isomorphic to an analytic twist by an element in the latter group.
\begin{theorem}[{\cite[Proposition 6.34]{kim_neron}}]
    This natural map $\Sha_{\text{\'et}}\to\Sha$ is injective and induces a bijection between the respective torsion subgroups.
\end{theorem}
\begin{remark}
    If the fibers are all integral, we actually know that $\Sha_{\text{\'et}}$ is torsion (\cite[Proposition 6.32]{kim_neron}).
\end{remark}
\begin{corollary}
    Suppose $\phi\in\Sha'$ is such that $X^\phi$ is projective, then $\phi\in\Sha_{\text{\'et}}$.
\end{corollary}

\subsection{Twists and \texorpdfstring{$C^\infty$}{Cinfty}-sections}\label{ss:twistsection}
Let us now explain why Sacc\`a's conjecture is true given the following extra condition on the fibration and its twist:
\begin{definition}
    A $C^\infty$-section of a fibration $\pi:X\to B$ between smooth varieties is a $C^\infty$-map $s:B\to X$ between the underlying real manifolds such that $\pi\circ s=\operatorname{id}_B$.
\end{definition}
\begin{lemma}
    Suppose $\pi:X\to B$ admits a $C^\infty$-section $s$, then every fiber of it has at least one reduced component.
\end{lemma}
\begin{proof}
    For any $b\in B$, the derivative $\mathrm D\pi_{s(b)}$ admits a right inverse given by $\mathrm Ds_b$.
    Hence $\mathrm D\pi_{s(b)}$ is surjective.
    This means that $\pi$ is smooth at $s(b)$, so the component of $X_b$ containing $s(b)$ is reduced.
\end{proof}
\begin{theorem}[\cite{BDV}]
    Suppose $\pi:X\to B$ admits a $C^\infty$-section $s$.
    Then there is a degenerate twistor deformation of $X$ such that $s$ becomes holomorphic.
\end{theorem}
\begin{corollary}\label{twist_rat_section}
    Suppose $\pi:X\to B$ is projective and admits a $C^\infty$-section.
    Then there is some torsion $\psi\in\Sha^0$ such that $X^\psi$ admits a holomorphic section.
\end{corollary}
\begin{proof}
    By the preceding theorem and Theorem \ref{abashevarogov}, we can find $\psi\in\Sha^0$ such that $\pi^\psi$ has a holomorphic section.
    But then $X^\psi$ must be projective (see \cite[Lemma 5.17]{AR21}), so $\psi$ is torsion by Theorem \ref{projectiveimpliestorsion}.
\end{proof}
Let us now give a positive answer to Sacc\`a's conjecture under some assumptions.
\begin{proposition}\label{conjecture_w_section}
    Suppose we have a projective Lagrangian fibration $\pi:X\to B$ and a torsion class $\phi\in\Sha_{X/B}$, such that both $\pi,\pi^\phi$ admit $C^\infty$-sections.
    Then $X$ is deformation-equivalent to $X^\phi$.
\end{proposition}
The proof relies on the following observation:
\begin{proposition}
    Suppose we have a projective Lagrangian fibration $\pi:X\to B$ with a rational section and an \'etale Tate-Shafarevich class $\phi\in\Sha_{\text{\'et}}$ such that $\pi^\phi:X^\phi\to B$ also has a rational section, then $X$ is birational to $X^\phi$.
\end{proposition}
\begin{proof}
    Remove from the base the (common) discriminant locus of the fibrations, as well as the indeterminacy loci of the rational sections.
    Denote by $B^0$ the resulting open locus, and $A,A^\phi$ the preimages of $B^0$ under $\pi,\pi^\phi$ respectively.

    Then $A,A^\phi$ are related by an \'etale Tate-Shafarevich twist over $B$.
    Since $A$ has a section, it can be given a canonical structure of an abelian scheme $\mu:A\times_{B^0}A\to A$ over $B^0$ with the section being the identity.
    But then $\underline{\operatorname{Aut}}_{A/B^0}^0\cong(A,\mu)^{\rm an}$ where the latter acts by translation.
    If we let $P\to B$ be the group scheme associated to $X\to B$ as in Theorem \ref{group_scheme_action}, then this shows that $P|_{B^0}\cong (A,\mu)$.
    Therefore $A^\phi\to B^0$ is a twist of $A\to B^0$ by the element
    \[\phi|_{B^0}\in H^1_{\text{\'et}}(B^0,P|_{B^0})=H^1_{\text{\'et}}(B^0,(A,\mu)).\]

    The group law $\mu:A\times_{B^0}A\to A$ gives rise to an action $A\times_{B^0}A^\phi\to A^\phi$ by \'etale descent, making $A^\phi$ a torsor for $(A,\mu)$.
    But $A^\phi\to B^0$ has a section, so this action gives an isomorphism $A\cong A^\phi$ as $B^0$-schemes.
    Therefore $X$ is birational to $X^\phi$.
\end{proof}
\begin{proof}[Proof of Proposition \ref{conjecture_w_section}]
    By Corollary \ref{twist_rat_section}, we can find torsion $\psi,\psi'\in\Sha$ such that $X$ is deformation-equivalent to $X^\psi$, $X^\phi$ is deformation-equivalent to $(X^\phi)^{\psi'}=X^{\phi+\psi'}$ (note also Remark \ref{ts_twice}), and both $X^{\psi}$ and $X^{\phi+\psi'}$ are projective and admit sections.

    Now we have $X^{\phi+\psi'}=(X^\psi)^{\phi+\psi'-\psi}$.
    Since $\phi+\psi'-\psi$ is torsion, the preceding proposition shows that $X^\psi$ is birational to $X^{\phi+\psi'}$.
    Since they are both irreducible hyper-K\"ahler varieties, this means that they are deformation-equivalent by \cite[Theorem 4.6]{huy}.
    Therefore $X$ and $X^\phi$ are deformation-equivalent.
\end{proof}
\bibliographystyle{plain}
\bibliography{ref}
\end{document}